\documentclass[pdflatex,sn-mathphys,Numbered]{sn-jnl}

\usepackage{graphicx}%
\usepackage{multirow}%
\usepackage{amsmath,amssymb,amsfonts}%
\usepackage{amsthm}%
\usepackage{mathrsfs}%
\usepackage[title]{appendix}%
\usepackage{xcolor}%
\usepackage{textcomp}%
\usepackage{manyfoot}%
\usepackage{booktabs}%
\usepackage{algorithm}%
\usepackage{algorithmicx}%
\usepackage{algpseudocode}%
\usepackage{listings}%
\usepackage{ulem}
\usepackage{subfig}
\usepackage{colortbl}
\usepackage{multirow}

\theoremstyle{thmstyleone}%
\newtheorem{theorem}{Theorem}
\theoremstyle{thmstyletwo}%
\newtheorem{remark}{Remark}%
\newtheorem{lemma}{Lemma}
\theoremstyle{thmstylethree}%

\begin{document}

\title[DNWR and NNWR Methods for PDEs with Time Delay]{Dirichlet-Neumann Waveform Relaxation Method for Hyperbolic PDE with Time Delay in Multiple Subdomains}


\author[1]{\fnm{Bankim C.} \sur{Mandal}}\email{bmandal@iitbbs.ac.in}

\author*[2]{\fnm{Deeksha} \sur{Tomer}}\email{a21ma09002@iitbbs.ac.in}

\affil[1]{\orgdiv{School of Basic Sciences}, \orgname{IIT Bhubaneswar}, \orgaddress{\city{Odisha}, \postcode{752050}, \country{India}}}

\affil[2]{\orgdiv{School of Basic Sciences}, \orgname{IIT Bhubaneswar}, \orgaddress{\city{Odisha}, \postcode{752050}, \country{India}}}


 \abstract{
 Hyperbolic partial differential equations (PDEs) with time delay are essential mathematical tools used to model numerous physical systems where wave propagation or oscillations depend heavily on their historical states. As these applications scale in physical complexity, efficient parallel computing techniques become essential; however, developing highly scalable parallel solvers for delayed PDEs remains a significant computational challenge. To address this gap, this study extends the Dirichlet-Neumann Waveform Relaxation (DNWR) method to a multi-subdomain framework explicitly designed for solving time-delayed hyperbolic PDEs in parallel. We advance the underlying mathematical framework and provide a rigorous convergence analysis for both one-dimensional and two-dimensional spatial configurations. A central theoretical contribution of this work is the establishment of finite-step convergence for the multi-domain DNWR algorithm. These theoretical guarantees are firmly corroborated through comprehensive numerical experiments. Furthermore, a detailed comparative analysis against alternative domain decomposition techniques namely NNWR, OSWR, and Classical SWR, highlight the advantages of the proposed method. Finally, this approach provides a robust, highly parallelizable solver capable of efficiently handling complex hyperbolic PDEs with time delay.
 }




\maketitle

\section{Introduction}

The study of physical systems often reveals that the rate of change of a process depends not only on its current state but also on its history. When these systems involve wave propagation, vibrations, or oscillations, they are mathematically captured by Hyperbolic Partial Differential Equations (PDEs) with time delay. We encounter many phenomena that are modeled by such equations; for example, thermo-acoustic combustion systems involve wave PDEs with feedback delay \cite{Wolfgang} to model instability caused by delayed interactions between flame dynamics and acoustics. Similar delayed interactions are prevalent in control theory \cite{nicaise2006stability}. Wave equations incorporating delay terms play a critical role in modeling oscillatory processes \cite{cui2003some} and are particularly essential for describing physical systems with aftereffects \cite{Liu}.

Despite their widespread application, a significant gap still exists in the comprehensive analysis of parallel methods for PDEs involving time delays. Parallel computing techniques \cite{gander, ganderparabolic} play a vital role in addressing the large-scale, complex numerical systems that frequently arise in physics and engineering. Among these, domain decomposition methods are particularly prominent due to their divide-and-conquer approach, where the computational domain is partitioned into multiple overlapping or non-overlapping subdomains. The primary motivation for employing these methods stems from the sheer size of the underlying systems, yet efficiently parallelizing delayed hyperbolic PDEs remains a formidable challenge.

To address this, we utilize the Dirichlet-Neumann Waveform Relaxation (DNWR) method, which belongs to the class of non-overlapping domain decomposition techniques. DNWR \cite{ganderparabolic, gander} is the extension of the steady-state Dirichlet-Neumann algorithm \cite{Bjorstad} for evolution problems. After each iteration, the interface condition is updated by forming a convex combination of the values obtained from the Dirichlet and Neumann solves averaging the old and new iterates along the interface. The algorithm proceeds iteratively until the solution attains smoothness across the entire domain. This approach naturally supports the use of different numerical schemes in different subdomains, making it highly suitable for problems with heterogeneous physical properties \cite{sana2024convergence, gander2026optimized}. Additionally, an important advantage of the DNWR method is its ability to act as a direct solver when the time window length is chosen appropriately.

In this work, we extend the implementation of the substructuring waveform relaxation method, specifically DNWR, from the two-subdomain case \cite{mandal2025dirichlet} to a multi-subdomain framework for solving Hyperbolic Delay PDEs. To illustrate this, the model problem we consider is a linear wave equation having a constant time delay, as given in \cite{Rodriguez}:$$w_{tt}=c^2w_{xx}+ \lambda w(x,t-\tau)+f(x,t),\  t>\tau, x\in\Omega\subset \mathbb{R}^d$$with initial conditions,$$w(x,t)=\phi (x,t), w_t(x,t)=\psi(x,t),\  -\tau\leq t\leq 0, x\in\Omega$$and Dirichlet boundary conditions,
\begin{equation}
\label{eq_1}
 w(x,t)=g(t), t\geq 0, x\in\partial\Omega
 \end{equation}
where $c$ denotes the propagation speed of the waves, and $\lambda$ represents a free parameter. Building upon this general formulation, a key advancement of our work is extending the DNWR framework and its rigorous convergence analysis beyond one-dimensional configurations to explicitly handle two-dimensional spatial domains.

\vspace{10pt} \noindent
The primary contributions of this work are summarized as follows:
\begin{itemize}
    \item 

We investigate the convergence of the DNWR method for one-dimensional problems in a multi-subdomain setting.

\item We extend the convergence analysis of the DNWR method to two-dimensional problems within a multi-subdomain framework.

\item We conduct a comprehensive comparative analysis to evaluate the performance of DNWR against alternative frameworks, namely NNWR, OSWR, and Classical SWR.
\end{itemize}

The remainder of the paper is organized as follows. Section \ref{sec_2} presents the formulation of the DNWR algorithm in a one-dimensional multidomain setting, followed by its convergence analysis in Section \ref{sec_3}. Section \ref{sec_4} details the DNWR formulation along with its convergence analysis in a 2D multi-domain setup. Numerical experiments demonstrating the efficiency of the algorithm within this multidomain framework are provided in Section \ref{sec_5}. Finally, Section \ref{sec_6} summarizes the main results of this paper.

\section{DNWR Formulation for multiple subdomain case}\label{sec_2}
For the implementation of DNWR in multiple subdomains, we have various arrangements depending upon the chosen transmission condition on the interface boundary. Different arrangements are shown in Fig. \ref{fig:arr_3}, and Fig. \ref{fig:arr_1_2}. Through numerical experiments on the reaction-diffusion equation with time delay, we compared Arrangements 1, 2, and 3, and found that Arrangement 3 outperforms the others in terms of iteration efficiency \cite{mandal2026dirichlet}. Also for the classical heat equation, the third arrangement gives us a better convergence result; see \cite{gander2021dirichlet}. We therefore focus on studying the convergence behavior of Arrangement 3, although a similar approach can be applied to Arrangements 1 and 2.
The domain $\Omega$ is divided into $N$ multiple subdomains $\Omega_i=(x_{i-1},x_i)$. The Dirichlet subproblem is initially solved within the middle subdomain, followed sequentially by the Dirichlet--Neumann subproblems in the adjacent subdomains.  
(Refer to Figure~\ref{fig:arr_3}). This procedure is iteratively applied for \( k = 1, 2, \ldots \) for all subdomains by introducing initial guesses $h_i^0$ along interfaces. Choosing the middle subdomain $\Omega_m$ for $m = \lceil{N/2\rceil}$, we solve

 \begin{equation}\label{eq_2}
\left\{\begin{array}{rl}
\partial_{tt} e_{m}^k-c ^2\partial_{xx} e_{m}^k-\lambda e_{m}^k(x, t-\tau )&=0, \ (x, t)\in \Omega_{m}\times (0, T),  \\ 
  e_{m}^k(x, t)&=0, \  (x, t)\in \Omega_{m}\times (-\tau, 0), \\ 
 \partial_t e_{m}^k(x, t)&=0, \  (x, t)\in \Omega_{m}\times (-\tau, 0), \\
e_{m+1}^k(x_{m-1}, t)&=h^{k-1}_{m-1},  \  t\in(0, T),\\ 
e_{m+1}^k(x_{m},t)&=h_{m}^{k-1}(t),  \  t\in(0, T),
\end{array}\right.
\end{equation}\\ 
subsequently, for each $\ell = 1, 2, \ldots , m-1$, the computation proceeds with $i = m-\ell$ and $j = m + \ell$, continuing until 
$j = N$ for even $N$,
   \begin{equation}
\left\{\begin{array}{rl}\label{eq_3}
\partial_{tt} e_i^k-c ^2\partial_{xx} e_i^k-\lambda e_i^k(x, t-\tau )&=0, \ (x, t)\in \Omega_i\times (0, T),  \\ 
  e_i^k(x, t)&=0, \  (x, t)\in \Omega_i\times (-\tau, 0), \\ 
 \partial_t  e_i^k(x, t)&=0, \  (x, t)\in \Omega_i\times (-\tau, 0), \\ 
e_i^k(x_{i-1} , t)&=h^{k-1}_{i-1}(t) \  t\in(0, T),\\
\partial_x e_i^k(x_i,t)&=\partial_x e_{i+1}^k(x_i,t), \  t\in(0, T),
\end{array}\right.
\end{equation}\\
\begin{equation}\label{eq_4}
\left\{\begin{array}{rl}
\partial_{tt} e_j^k-c ^2\partial_{xx} e_j^k-\lambda e_j^k(x, t-\tau )&=0, \ (x, t)\in \Omega_j\times (0, T),  \\ 
  e_j^k(x, t)&=0, \  (x, t)\in \Omega_j\times (-\tau, 0), \\ 
\partial_t   e_j^k(x, t)&=0, \  (x, t)\in \Omega_j\times (-\tau, 0), \\ 
\partial_x e_j^k(x_{j-1},t)&=\partial_x e_{j-1}^k(x_{j-1},t), \  t\in(0, T), \\
e_j^k(x_{j} , t)&=h^{k-1}_j(t) \  t\in(0, T).
\end{array}\right.
\end{equation}\\
The update conditions along the interfaces are:\\
\begin{equation*}
h_i^{k}(t)=\theta e_{i}^{k}(x_i,t) +(1-\theta) h_i^{k-1}( t);\  1 \leq i < m \\
\end{equation*}
\begin{equation}\label{eq_5}
h_j^{k}(t)=\theta e_{j+1}^{k}(x_j,t) +(1-\theta) h_j^{k-1}( t);\ m\leq j \leq N-1,\\
\end{equation}
where $\theta\in(0,1]$ is a relaxation parameter.

   \begin{figure}[!h]
     \centering
     \includegraphics[width=0.6 \linewidth]{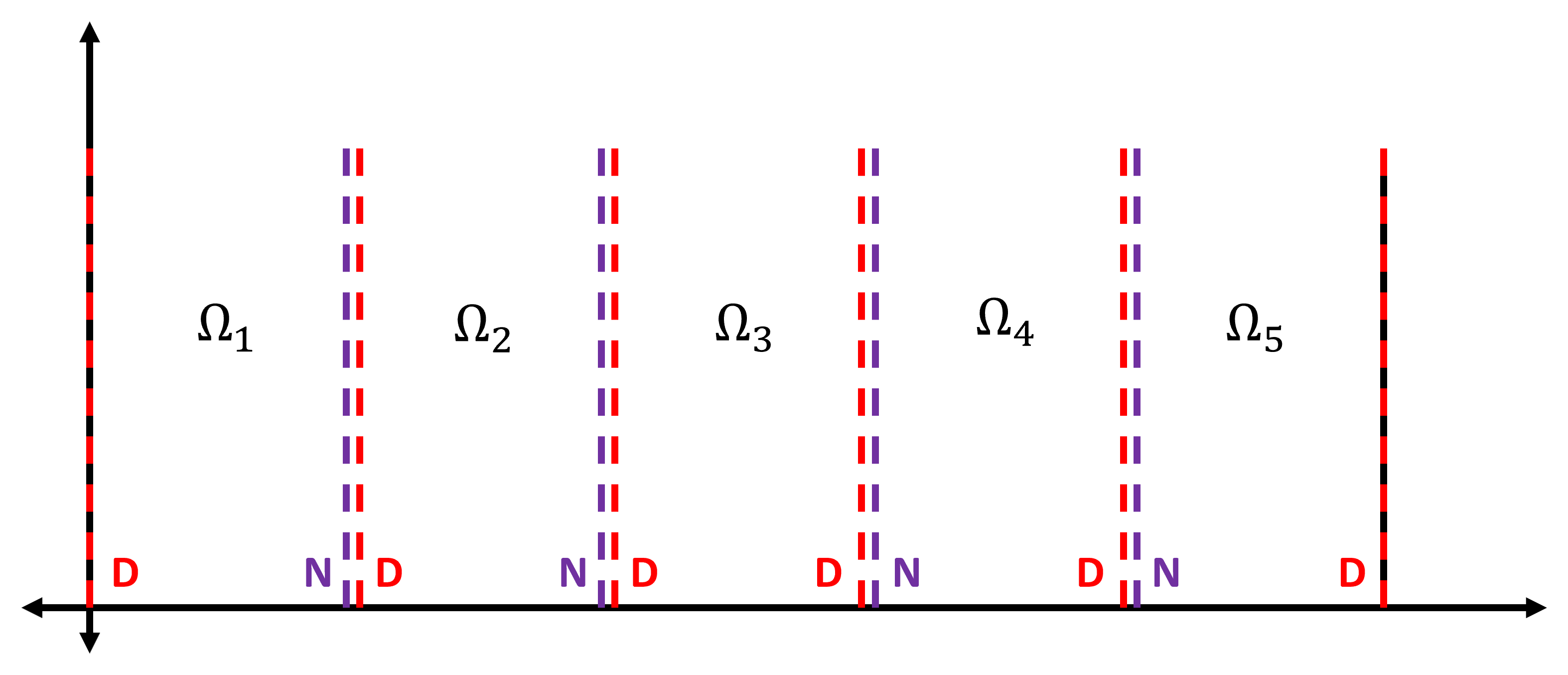}
     \caption{Boundary conditions for arrangement~3 in DNWR for multisubdomain case. }
     \label{fig:arr_3}
 \end{figure}
 
 \begin{figure}[!h]
    \centering
    \subfloat[Arrangement 1]{\includegraphics[width=0.462\linewidth]{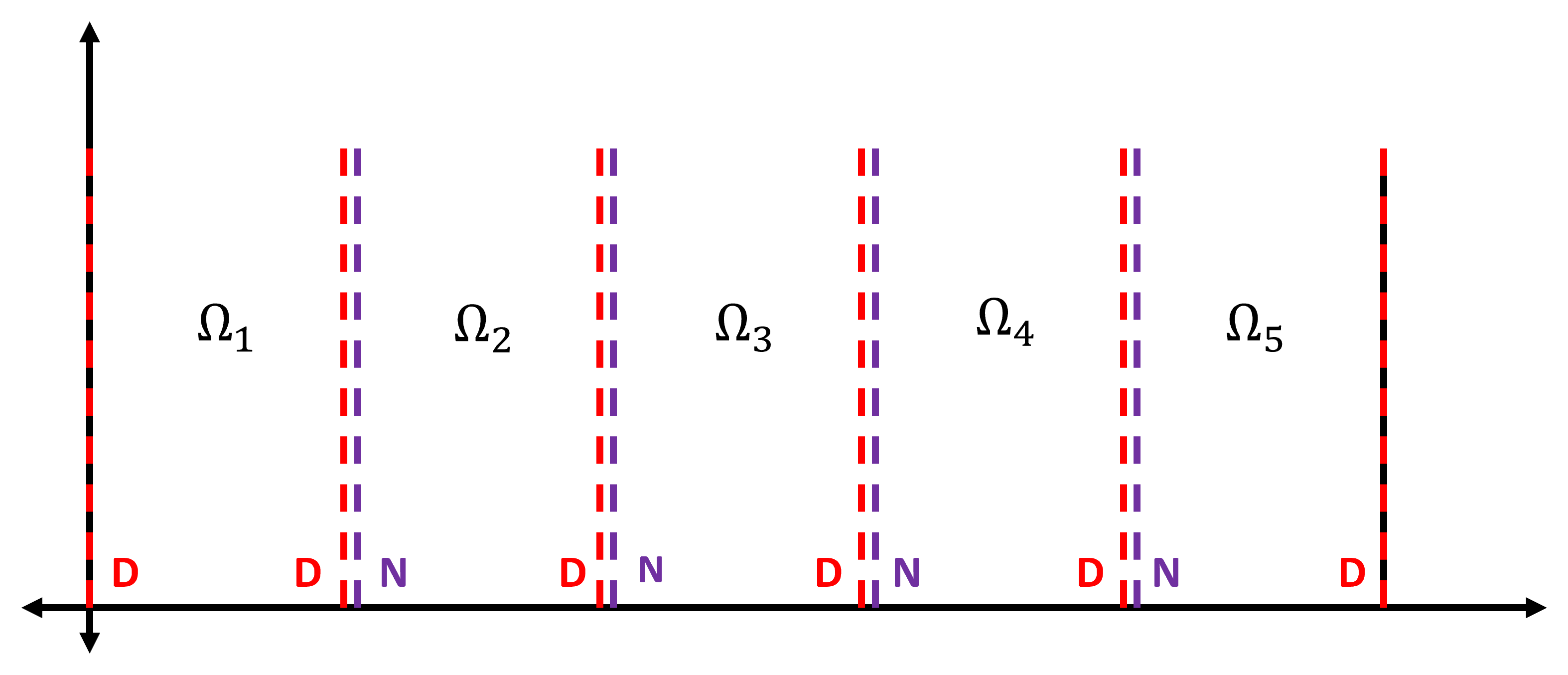}}
    \qquad
    \subfloat[Arrangement 2]{\includegraphics[width=0.462\linewidth]{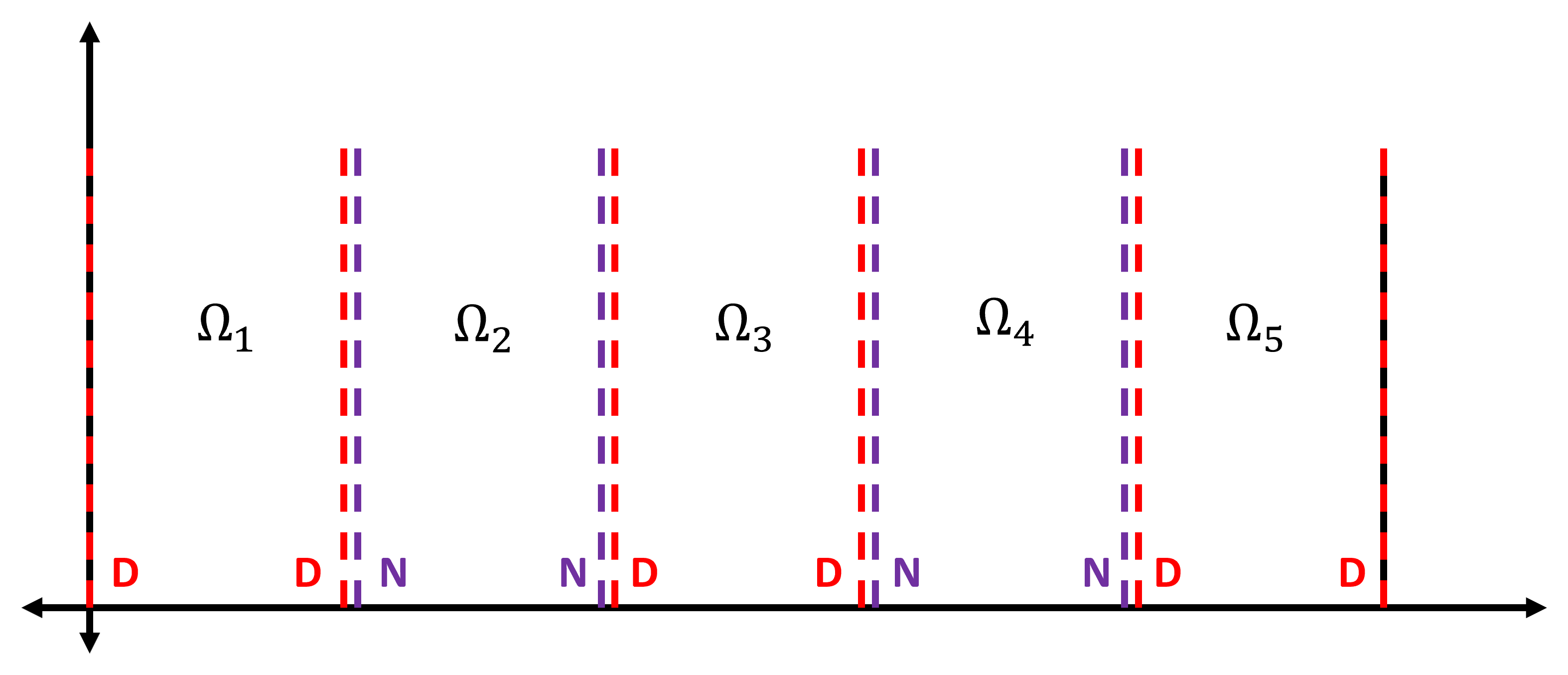}}
    \caption{Setup of boundary conditions in DNWR for multisubdomain case.}
    \label{fig:arr_1_2}
\end{figure}

 \section{DNWR Convergence Analysis}\label{sec_3}
 The convergence analysis in this section is primarily based on Laplace transform method. For the DNWR algorithm, we employ an optimal relaxation parameter of $1/2$, as established in previous studies \cite{mandal2025dirichlet,gander,ganderparabolic}. Therefore, the subsequent analysis proceeds under the assumption that $\theta = 1/2$. Before moving on to the main results, we first introduce essential preliminaries.\\
 \noindent\textbf{Definition (Convolution):}  
Let $f(t)$ and $g(t)$ be piecewise continuous functions. Their convolution is defined by,
\begin{equation}\label{convolution}
    (f * g)(t) = \mathcal{L}^{-1}\{F(s)G(s)\} =\int_{0}^{t} f(\tau)\, g(t-\tau)\, d\tau,
\end{equation}
where $F(s)$ and $G(s)$ denote the Laplace transforms of $f(t)$ and $g(t)$, respectively.\\

 \noindent \textbf{Second Translation Theorem:}  
Let $P(s)$ be the Laplace transform of a function $p(t)$, i.e., $P(s)=\mathcal{L}\{p(t)\}$. 
For $\beta \geq 0$, the inverse Laplace transform of $e^{-\beta s}P(s)$ is expressed as
\begin{equation}\label{eq_7}
    \mathcal{L}^{-1}\{e^{-\beta s}P(s)\} = H(t-\beta)\, p(t-\beta)=
    \begin{cases}
p(t - \beta), & t > \beta, \\
0, & t \leq \beta,
\end{cases},
\end{equation}
where $H(t-\beta)$ denotes the Heaviside step function.

 \noindent \textbf{Efros Theorem:} Efros theorem \cite{sana2023dirichlet} states that if $\hat{p}(s)$ and $\hat{r}(s)e^{-q(s)\tau}$ are the Laplace transforms of $p(t)$ and $r(t,\tau)$ with respect to $t$, where $\tau$ is treated as a parameter, then
\begin{equation} \label{efros}
 \mathcal{L}^{-1}\{\hat{p}(q(s))\hat{r}(s)\}=\int_{0}^{\infty} r(t,\tau)\, p(\tau)\, d\tau .   
\end{equation}

\begin{theorem}\label{thm1}
  The DNWR algorithm for \( \theta = 1/2 \) defined by equations (\ref{eq_2}-\ref{eq_5}) converges in at most \( k + 1 \) iterations for multiple subdomains, provided the time window length \( T \) satisfies:
\[
T/k \leq d_{\min}/c,
\]
where:
 \( d_{\min} \) is the minimal subdomain width, \( c \) is the wave speed.

 \end{theorem}
 \begin{proof}

To initiate the analysis, we apply the Laplace transform in the time variable \( t \) to the homogeneous Dirichlet subproblem described in equation (\ref{eq_2}). Let \( \hat{e}^k_{m}(x, s) \) represent the Laplace-transformed solution for the error equation within subdomain \( \Omega_{m} \) at the \( k \)-th iteration.
 The transformed subdomain problem from equation \eqref{eq_2} becomes:
$$
s^2 \hat{e}^k_{m} - c^2 \frac{ \hat{e}^k_{m}}{dx^2}-\lambda e^{-\tau s}\hat e_{m}^k = 0,
$$
with boundary conditions:
$$
\hat{e}^k_{m}(x_{m-1}, s) = \hat{h}^{k-1}_{m-1}(s), \quad 
\hat{e}^k_{m}(x_{m}, s) = \hat{h}^{k-1}_{m}(s).
$$
Define $\alpha _i= \sinh(d_i\frac{\sqrt{s^2-\lambda e^{-\tau s}}}{c})$ and $\beta  _i= \cosh(d_i\frac{\sqrt{s^2-\lambda e^{-\tau s}}}{c})$
where \( d_{i} = x_{i} - x_{i-1} \) is the width of subdomain \( \Omega_{i} \).
Then the solution within subdomain \( \Omega_{m} \) is given by:
{\fontsize{8}{9.5}\selectfont
$$
\hat{e}^k_{m}(x, s) = 
\frac{1}{\alpha_{m}} \left[
\hat{h}^{k-1}_{m}(s) \sinh\left(\frac{\sqrt{s^2-\lambda e^{-\tau s}}}{c}(x - x_{m-1})\right) + 
\hat{h}^{k-1}_{m-1}(s) \sinh\left(\frac{\sqrt{s^2-\lambda e^{-\tau s}}}{c}(x_{m} - x)\right)
\right].
$$}
Let us define \( \zeta=\frac{\sqrt{s^2-\lambda e^{-\tau s}}}{c}\); then $\alpha_i=\sinh(d_i \zeta)$, $\beta_i=\cosh(d_i \zeta)$, $\hat w_i^k(s)=\partial_x\hat e^k_{i+1}(x_i,s)$, $\hat w_j^k(s)=-\partial_x\hat e_j^k(x_j,s)$ and the solutions to the subproblems are expressed as follows: for \(1 \leq i < m\) and \(m+1 \leq j \leq N\)
\begin{align*}
\hat{e}_i^k(x, s) 
&= \frac{1}{\beta_i \zeta}\hat{w}^k_i 
   \sinh\!\left((x - x_{i-1})\zeta\right)
   + \frac{1}{\beta_i} \hat{h}^{k-1}_{i-1} 
   \cosh\!\left((x_i - x)\zeta\right), \\[6pt]
\hat{e}_j^k(x, s) 
&= \frac{1}{\beta_j} \hat{h}^{k-1}_j 
   \cosh\!\left((x - x_{j-1}) \zeta\right)
   + \frac{1}{\beta_j\zeta}\hat{w}^k_{j-1} 
   \sinh\!\left((x_j - x)\zeta\right).
\end{align*}  
For \(\theta = 1/2\), the update conditions become:

\[
\hat{w}^k_i = -\zeta \frac{\alpha_{i+1}}{\beta_{i+1}} \hat{h}^{k-1}_i + \frac{1}{\beta_{i+1}} \hat{w}^k_{i+1}, \quad 1 \leq i < m-1,
\]

\[
\hat{w}^k_{m-1} = -\zeta \frac{\beta_{m}}{\alpha_{m}} \hat{h}^{k-1}_{m-1} + \zeta  \frac{1}{\alpha_{m}} \hat{h}^{k-1}_{m}, \quad
\hat{w}^k_{m} = \zeta \frac{1}{\alpha_{m}} \hat{h}^{k-1}_{m-1} - \zeta  \frac{\beta_{m}}{\alpha_{m}} \hat{h}^{k-1}_{m},
\]

\[
\hat{w}^k_j = \frac{1}{\beta_j} \hat{w}^k_{j-1} - \zeta  \frac{\alpha_j}{\beta_j} \hat{h}^{k-1}_j, \quad m+1 \leq j \leq N-1.
\]
\[
\hat h^k_i = \frac{1}{2\beta_{i}} \hat{h}^{k-1}_{i-1} + \frac{1}{2} \hat{h}^{k-1}_{i} + \frac{\alpha_i}{2\beta_i}\frac{\hat{w}^k_i }{\zeta}, \quad 1 \leq i < m,
\]

\[
\hat{h}^k_j = \frac{\alpha_{j+1}}{2\beta_{j+1}} \frac{\hat{w}^k_j}{\zeta} + \frac{1}{2} \hat{h}^{k-1}_j + \frac{1}{2\beta_{j+1}} \hat{h}^{k-1}_{j+1}, \quad m \leq j \leq N-1.
\]
We now define:
\[
\bar{h}^k_i := \beta_i \hat{h}^k_i, \quad \bar{w}^k_i := \frac{\hat{w}^k_i\alpha_i}{\zeta}, \quad 1 \leq i < m,
\]
\[
\bar{h}^k_j := \beta_{j+1} \hat{h}^k_j, \quad \bar{w}^k_j := \frac{\hat{w}^k_j \alpha_{j+1}}{\zeta}, \quad m \leq j \leq N,
\]
with boundary values \( \beta_0 = \beta_{N+1} = 1 \).
Then, 
we get update equations for \( \bar{h}_k \):
\[
\bar{h}^k_i = \frac{1}{2\beta_{i-1}} \bar{h}^{k-1}_{i-1} + \frac{1}{2} \bar{h}^{k-1}_{i} + \frac{1}{2} \bar{w}^k_i, \quad 1 \leq i < m,
\]
\[
\bar{h}^k_j = \frac{1}{2} \bar{w}^k_j + \frac{1}{2} \bar{h}^{k-1}_j + \frac{1}{2\beta_{j+2}} \bar{h}^{k-1}_{j+1}, \quad m \leq j \leq N-1.
\]
The update equations for \( \bar{w}_k \):

\[
\bar{w}^k_i = -\frac{\alpha_i \alpha_{i+1}}{ \beta_i \beta_{i+1}} \bar{h}^{k-1}_i + \frac{\alpha_i}{\alpha_{i+1}\beta_{i+1}} \bar{w}^k_{i+1}, \quad 1 \leq i \leq m - 2,
\]
\[
\bar{w}^k_{m-1} = -\frac{\alpha_{m-1} \beta_{m} }{\beta_{m-1} \alpha_{m}} \bar{h}^{k-1}_{m-1} + \frac{\alpha_{m-1}}{\alpha_{m} \beta_{m+1} } \bar{h}^{k-1}_{m},
\]
\[
\bar{w}^k_{m} = \frac{\alpha_{m+1}}{ \beta_{m-1} \alpha_{m}} \bar{h}^{k-1}_{m-1} - \frac{\beta_{m} \alpha_{m+1}}{\alpha_{m} \beta_{m+1}}  \bar{h}^{k-1}_{m},
\]
\[
\bar{w}^k_j = \frac{\alpha_{j+1}}{\beta_j \alpha_j} \bar{w}^k_{j-1} - \frac{\alpha_j \alpha_{j+1}}{\beta_j \beta_{j+1}} \bar{h}^{k-1}_j, \quad m+1 \leq j \leq N-1.
\]
Thus, the Matrix representation of interface error where $\alpha _{i,j}=\sinh((d_i-d_j)\zeta)$ and $\beta_{i,j}=\cosh((d_i-d_j)\zeta)$ is given as:
\[
\begin{pmatrix}
\bar h_1^{k}\\
\vdots\\
\bar h_{m-1}^{k}\\
\bar h_{m}^{k}\\
\vdots\\
\bar h_{N-1}^{k}
\end{pmatrix}
=
M
\begin{pmatrix}
\bar h_1^{k-1}\\
\vdots\\
\bar h_{m-1}^{k-1}\\
\bar h_{m}^{k-1}\\
\vdots\\
\bar h_{N-1}^{k-1}
\end{pmatrix}, \quad \text{with} \; M=
\begin{pmatrix}
P_1 & P_2\\
P_3 & P_4
\end{pmatrix}_{(N-1)\times (N-1)},
\]
and
\[
P_1=
\begin{bmatrix}
\dfrac{\beta_{1,2}}{2\beta_1\beta_2}
&
\dfrac{-\alpha_1\alpha_3}{2\beta_2^2\beta_3}
& \cdots &
\dfrac{-\alpha_1\alpha_{m-1}}{2\beta_2\cdots\beta_{m-3}\beta_{m-2}^2\beta_{m-1}}
&
\dfrac{-\alpha_1\beta_{m}}{2\beta_2\cdots\beta_{m-2}\beta_{m-1}^2\alpha_{m}}
\\[6pt]
\dfrac{1}{2\beta_1}
&
\dfrac{\beta_{2,3}}{2\beta_2\beta_3}
& \cdots &
\dfrac{-\alpha_2\alpha_{m-1}}{2\beta_3\cdots\beta_{m-3}\beta_{m-2
}^2\beta_{m-1}}
&
\dfrac{-\alpha_2\beta_{m}}{2\beta_3\cdots\beta_{m-2}\beta_{m-1}^2\alpha_{m}}
\\
0 & \ddots & \ddots & \vdots & \vdots\\
\vdots & 0 &
\dfrac{1}{2\beta_{m-3}}
&
\dfrac{\beta_{m-2,m-1}}{2\beta_{m-2}\beta_{m-1}}
&
\dfrac{-\alpha_{m-2}\beta_{m}}{2\beta_{m-1}^2\alpha_{m}}
\\
0 & \cdots & 0 &
\dfrac{1}{2\beta_{m-2}}
&
\dfrac{\alpha_{m,m-1}}{2\alpha_{m}\beta_{m-1}}
\end{bmatrix},
\]
\[
\begin{aligned}
P_2=
\begin{bmatrix}
\dfrac{\alpha_1}{2\beta_2\cdots\beta_{m-1}\alpha_{m}\beta_{m+1}} & 0 & \cdots &\cdots & 0\\
\dfrac{\alpha_2}{2\beta_3\cdots\beta_{m-1}\alpha_{m}\beta_{m+1}} & \vdots & \vdots &\vdots & \vdots\\
\vdots & \vdots & \vdots &\vdots & \vdots\\
\dfrac{\alpha_{m-2}}{2\beta_{m-1}\alpha_{m}\beta_{m+1}} & \vdots & \vdots &\vdots & \vdots\\
\dfrac{\alpha_{m-1}}{2\alpha_{m}\beta_{m+1}} & 0 & \cdots &\cdots & 0
\end{bmatrix},
\qquad
P_3=
\begin{bmatrix}
0 & \cdots &\cdots & 0 &
\dfrac{\alpha_{m+1}}{2\beta_{m-1}\alpha_{m}}
\\
\vdots & \vdots &\vdots & \vdots &
\dfrac{\alpha_{m+2}}{2\beta_{m-1}\alpha_{m}\beta_{m+1}}
\\
\vdots & \vdots & \vdots &\vdots & \vdots\\
\vdots & \vdots &\vdots & \vdots &
\dfrac{\alpha_{N-1}}{2\beta_{m-1}\alpha_{m}\beta_{m+1}\cdots \beta_{N-2}}
\\
0 & \cdots & \cdots & 0 &
\dfrac{\alpha_{N}}{2\beta_{m-1}\alpha_{m}\beta_{m+1}\cdots\beta_{N-1}}
\end{bmatrix},
\end{aligned}
\]

\[
P_4=
\begin{bmatrix}
\dfrac{\alpha_{m,m+1}}{2\alpha_{m}\beta_{m+1}}
&
\dfrac{1}{2\beta_{m+2}}
& 0 & \cdots & 0
\\
\dfrac{-\alpha_{m+2}\beta_{m}}{2\alpha_{m}\beta_{m+1}^2}
&
\dfrac{\beta_{m+1,m+2}}{2\beta_{m+1}\beta_{m+2}}
&
\dfrac{1}{2\beta_{m+3}}
& 0 & \vdots
\\
\vdots & \vdots & \ddots & \ddots & 0\\
\dfrac{-\alpha_{N-1}\beta_{m}}{2\alpha_{m}\beta_{m+1}^2\beta_{m+2}\cdots\beta_{N-2}}
&
\dfrac{-\alpha_{N-1}\alpha_{m+1}}{2\beta_{m+1}\beta_{m+2}^2\beta_{m+3}\cdots\beta_{N-2}}
&
\ddots
&
\dfrac{\beta_{N-2,N-1}}{2\beta_{N-2}\beta_{N-1}}
&
\dfrac{1}{2\beta_{N}}
\\
\dfrac{-\alpha_{N}\beta_{m}}{2\alpha_{m}\beta_{m+1}^2\beta _{m+2}\cdots\beta_{N-1}}
&
\dfrac{-\alpha_{N}\alpha_{m+1}}{2\beta_{m+1}\beta_{m+2}^2\beta_{m+3}\cdots\beta_{N-1}}
&
\cdots
&
\dfrac{-\alpha_{N-2}\alpha_{N}}{2\beta_{N-2}\beta_{N-1}^2}
&
\dfrac{\beta_{N-1,N}}{2\beta_{N-1}\beta_{N}}
\end{bmatrix}.
\]
As a result, the update condition is modified to,
\begin{equation}\label{eq_9}
\hat h_i^k(s)=\sum_{l=i-1}^{m}\hat\gamma_{i,l}\hat h_l^{k-1}(s), \quad 1\leq i< m; \qquad
\hat h_j^k(s)=\sum_{l=m-1}^{j+1}\hat\gamma_{j,l}\hat h_l^{k-1}(s), \quad m\leq j\leq N-1.
\end{equation}
The coefficients $\hat{\gamma}_{i,j}$ are defined such that $\hat{\gamma}_{1,0} = \hat{\gamma}_{N-1,N} = 0$, and for $1 \leq i < m-1$ and $i + 1 \leq l < m-1$, the following holds:

\begin{align*}
\hat{\gamma}_{i, i-1} = \frac{1}{2\beta_i},
\hat{\gamma}_{i, i} = \frac{\beta_{i, i+1}}{2\beta_i \beta_{i+1}},
\hat{\gamma}_{i, l} = -\frac{\alpha_i \alpha_{l+1}}{2\beta_i \beta_{i+1} \cdots \beta_{l+1}}, \\
\hat{\gamma}_{i, m-1} = -\frac{\alpha_i \beta_{m}}{2\beta_i \cdots \beta_{m-1} \alpha_{m}}, 
\hat{\gamma}_{i, m} = \frac{\alpha_i}{2\beta_i \cdots \beta_{m-1} \alpha_{m}}, \\
\text{ and for } m < l \leq j-1,\, m < j \leq N-1 \quad & \\
\hat{\gamma}_{j, j} = \frac{\beta_{j, j+1}}{2\beta_j \beta_{j+1}}, 
\hat{\gamma}_{j, l} = -\frac{\alpha_{j+1} \alpha_l}{2\beta_l \beta_{l+1} \cdots \beta_{j+1}}, 
\hat{\gamma}_{j, j+1} = \frac{1}{2\beta_{j+1}}, \\
\hat{\gamma}_{j, m} = -\frac{\alpha_{j+1} \beta_{m}}{2\alpha_{m} \beta_{m+1} \cdots \beta_{j+1}}, 
\hat{\gamma}_{j, m-1} = \frac{\alpha_{j+1}}{2\alpha_{m} \beta_{m+1} \cdots \beta_{j+1}}.
\end{align*}
Also,
\begin{align*}
\hat{\gamma}_{m-1, m-2} = \frac{1}{2\beta_{m-1}},
\hat{\gamma}_{m-1, m-1} = \frac{\alpha_{m,m-1}}{2\alpha_{m}\beta_{m-1} },
\hat{\gamma}_{m-1, m} = \frac{\alpha_{m-1}}{2 \alpha_{m} \beta_{m-1}}, \\
\hat{\gamma}_{m, m-1} = \frac{\alpha_{m+1}} {2\alpha_{m} \beta_{m+1}}, 
\hat{\gamma}_{m, m} = \frac{\alpha_{m,m+1}}{2\alpha_{m} \beta_{m+1}},
\hat{\gamma}_{m,m+1}=\frac{1}{2\beta_{m+1}.}
\end{align*}
By induction on \eqref{eq_9}, we can write for $1 \leq i \leq N-1$,

\begin{equation}\label{eq_10}
\hat{h}^k_i(s) = \sum_{j=1}^{N-1} \mu^n_{i,j} \left( \hat{\gamma}_{1,1}, \hat{\gamma}_{1,2}, \ldots, \hat{\gamma}_{N-1,N-2}, \hat{\gamma}_{N-1,N-1} \right) \hat{h}^{k-n}_j(s).
\end{equation}
The coefficients $\mu^n_{i,j}$ are either zero or homogeneous polynomials of degree $n$.
Subsequently, we express the hyperbolic functions as infinite series of exponential functions by applying the geometric series expansion, utilizing the identity

\begin{equation}
\cosh(z) = \frac{1}{2}(e^z + e^{-z}) = \frac{e^z}{2}(1 + e^{-2z}).
\end{equation}
This enables us to express the coefficients $\hat{\gamma}_{i,j}$ as follows: for $1 \leq i < m-1$ and $i + 1 \leq l < m-1$, the $(i,i)$-th entry is given by:
\begin{align*}
\hat\gamma_{i,i} 
&= \frac{\cosh\left((d_i - d_{i+1})\zeta\right)}{2 \cosh(d_i \zeta) \cosh(d_{i+1}\zeta)} \\
&= \frac{e^{(d_i - d_{i+1})\zeta} + e^{(d_{i+1} - d_i)\zeta}}{e^{d_i \zeta}(1 + e^{-2d_i \zeta}) \cdot e^{d_{i+1} \zeta}(1 + e^{-2d_{i+1}\zeta})} \\
&= \frac{e^{-2d_{i+1}\zeta} + e^{-2d_i \zeta}}{(1 + e^{-2d_i \zeta})(1 + e^{-2d_{i+1}\zeta})} \\
&= \left( e^{-2d_i \zeta} + e^{-2d_{i+1}\zeta} \right)
\left[
1 + \sum_{l=1}^{\infty} (-1)^l e^{-2ld_i \zeta} 
+ \sum_{n=1}^{\infty} (-1)^n e^{-2n d_{i+1}\zeta} \right. \\
&\quad \left.
+ \sum_{l=1}^{\infty} \sum_{n=1}^{\infty} (-1)^{l+n} e^{-2(l d_i + n d
_{i+1})\zeta}
\right].
\end{align*}
In a similar manner, the remaining coefficients can be expanded as follows:
\begin{align*}
\hat\gamma_{i,l} &= -\frac{\sinh(d_i \zeta)\sinh(d_{l+1} \zeta)}{2 \cosh(d_i \zeta)\cosh(d_{i+1} \zeta) \cdots \cosh(d_{l+1} \zeta)} \\
&= -2^{l - i - 1} e^{-(d_{i+1} + \cdots + d_l)\zeta} 
\left( 1 - e^{-2d_i \zeta} - e^{-2d_{l+1} \zeta} + e^{-2(d_i + d_{l+1})\zeta} \right)
\prod_{n=i}^{l+1} \left( 1 + e^{-2d_n \zeta} \right)^{-1}, \\[1em]
\hat{\gamma}_{i,i-1} &= \frac{1}{2 \cosh(d_i \zeta)} 
= e^{-d_i \zeta} \left( 1 + \sum_{l=1}^{\infty} (-1)^l e^{-2ld_i \zeta} \right), \\[1em]
\hat{\gamma}_{i,m-1} &= -\frac{\sinh(d_i \zeta)\cosh(d_{m} \zeta)}{2 \cosh(d_i \zeta) \cdots \cosh(d_{m-1} \zeta)\sinh(d_{m} \zeta)} \\
&= -2^{m - i - 2} e^{-(d_{i+1} + \cdots + d_{m-1})\zeta}
\left( 1 - e^{-2d_i \zeta} + e^{-2d_{m} \zeta} - e^{-2(d_i + d_{m})\zeta} \right)\\
&\left( 1 - e^{-2d_{m} \zeta} \right)^{-1}
\prod_{l=i}^{m-1} \left( 1 + e^{-2d_l \zeta} \right)^{-1},
\end{align*}
\begin{align*}
\hat{\gamma}_{i,m} 
&= \frac{\sinh(d_i \zeta)}{2 \cosh(d_i \zeta) \cdots \cosh(d_{m-1} \zeta) \sinh(d_{m} \zeta)} \\
&= 2^{m-i-1} e^{-(d_{i+1} + \cdots + d_{m})\zeta}
\left( 1 - e^{-2d_i \zeta} \right)
\left( 1 - e^{-2d_{m} \zeta} \right)^{-1}
\prod_{l=i}^{m-1} \left( 1 + e^{-2d_l \zeta} \right)^{-1}, \\[1em]
\hat{\gamma}_{m-1,m-1} 
&= \frac{\sinh((d_{m} - d_{m-1})\zeta)}{2 \cosh(d_{m-1} \zeta) \sinh(d_{m} \zeta)} \\
&= \left( e^{-2d_{m-1} \zeta} - e^{-2d_{m} \zeta} \right)\cdot\\
&\left[
1 + \sum_{l=1}^{\infty} (-1)^l e^{-2l d_{m-1} \zeta}
+ \sum_{n=1}^{\infty} e^{-2n d_{m} \zeta}
+ \sum_{l=1}^{\infty} \sum_{n=1}^{\infty} (-1)^l e^{-2(l d_{m-1} + n d_{m}) \zeta}
\right], \\[1em]
\hat{\gamma}_{m-1,m} 
&= \frac{\sinh(d_{m-1} \zeta)}{2 \cosh(d_{m-1} \zeta) \sinh(d_{m} \zeta)} \\
&= \left( e^{-d_{m} \zeta} - e^{-(2d_{m-1} + d_{m}) \zeta} \right)\cdot\\
&\left[
1 + \sum_{l=1}^{\infty} (-1)^l e^{-2l d_{m-1} \zeta}
+ \sum_{n=1}^{\infty} e^{-2n d_{m} \zeta}
+ \sum_{l=1}^{\infty} \sum_{n=1}^{\infty} (-1)^l e^{-2(l d_{m-1} + n d_{m}) \zeta}
\right].
\end{align*}
The remaining terms are similarly expanded using the geometric series. These resulting expressions enable us to reformulate Equation \eqref{eq_10} as
\[
\hat{h}^k_i(s) = \sum_{j=1}^{N-1} \hat{p}^k_{i,j}(s)\, \hat{h}^0_j(s),
\]
Each coefficient $\hat{p}^k_{i,j}(s)$ is a linear combination of exponential terms of the form $e^{-s\zeta}$, where $\zeta \geq \frac{k d_l}{c}$ for some $l \in \{1, 2, \ldots, N\}$. By applying the shifting property \eqref{eq_7} of the inverse Laplace transform,
we deduce that
\[
h^k_i(t) = h^0_j \left(t - \frac{k d_l}{c} \right) H \left(t - \frac{k d_l}{c} \right) + h^0_j \left(t - \frac{k d_l}{c} -\tau\right) H \left(t - \frac{k d_l}{c} -\tau  \right)+ \text{other terms}.
\]
for some $j \in \{1, 2, \ldots, N-1\}$ and $l \in \{1, 2, \ldots, N\}$, with all other terms vanishing when $t \leq \frac{k d_l}{c}$. Hence, if $T \leq \frac{k d_{\min}}{c}$, it follows that $h^k_i(t) = 0$ for all $i$, thereby concluding the proof.\\
 \end{proof}

 \section{DNWR Convergence in 2D}\label{sec_4}
In this section we analyse the efficiency of DNWR in 2-D spatial domain for wave PDEs with time delay for model problem \eqref{eq_1}. We consider again the error equation of the model problem due to the linearity. To formulate the DNWR algorithm, we decompose the spatial domain \( \Omega \) into a set of vertical strips defined as
$
\Omega_i = (x_{i-1}, x_i) \times (0, \pi), \quad \text{for } i = 1, \ldots, N,
$
where \( x_0 = l < x_1 < \cdots < x_N = L \). The width of each subdomain is denoted by \( d_i := x_i - x_{i-1} \), and we define the minimal subdomain width as $d_{\min} := \min_{1 \leq i \leq N} d_i$.

The DNWR algorithm in 2D looks like combination of cuboids (see Fig.\ref{2Dstrips}). For $m:=\lceil N/2 \rceil $, this procedure is iteratively applied for \( k = 1, 2, \ldots \) for all subdomains and are expressed as:
\begin{figure}
    \centering
    \includegraphics[width=0.9 \linewidth]{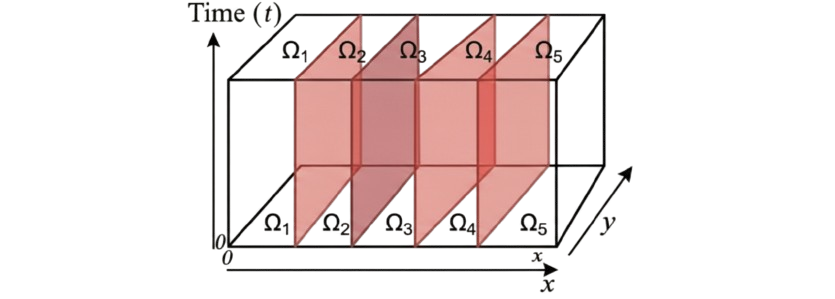}
   \caption{Evolution of 2D spatial strips over time as a decomposition of the total space-time volume.}
     \label{2Dstrips}
    \end{figure}

 \begin{equation}\label{eq_12}
\left\{\begin{array}{rl}
\partial_{tt} e_{m}^k-c ^2(\partial_{xx} e_{m}^k+\partial_{yy}e_{m}^k)-\lambda e_{m}^k(x, y, t-\tau )&=0, \ (x,y, t)\in \Omega_{m}\times (0, T),  \\ 
  e_{m}^k(x,y, t)&=0, \  (x,y, t)\in \Omega_{m}\times (-\tau, 0), \\ 
 \partial_t e_{m}^k(x, y, t)&=0, \  (x, y, t)\in \Omega_{m}\times (-\tau, 0), \\
e_{m}^k(x_{m-1},y, t)&=\kappa^{k-1}_{m-1}(y,t),  \  t\in(0, T),\\ 
e_{m}^k(x_{m},y ,t)&=\kappa_{m}^{k-1}(y,t),  \  t\in(0, T),\\
e_{m}^k(x,0,t)&=e_{m}^k(x,\pi,t)=0, \ t \in (0,T),
\end{array}\right.
\end{equation}

 and then for $1\leq i < m$,
   \begin{equation}
\left\{\begin{array}{rl}\label{eq_13}
\partial_{tt} e_i^k-c ^2(\partial_{xx} e_i^k+\partial_{yy} e_i^k)-\lambda e_i^k(x, y, t-\tau )&=0, \ (x, y, t)\in \Omega_i\times (0, T),  \\ 
  e_i^k(x, y, t)&=0, \  (x, y,t)\in \Omega_i\times (-\tau, 0), \\ 
 \partial_t  e_i^k(x, y, t)&=0, \  (x, y,t)\in \Omega_i\times (-\tau, 0), \\ 
e_i^k(x_{i-1}, y, t)&=\kappa^{k-1}_{i-1}(y,t), \  t\in(0, T),\\
\partial_x e_i^k(x_i, y, t)&=\partial_x e_{i+1}^k(x_i, y,t), \  t\in(0, T),\\
e_i^k(x,0,t)&=e_i^k(x,\pi,t)=0, \ \ t\in (0,T),
\end{array}\right.
\end{equation}
and for $m+1\leq j \leq N$ \\
\begin{equation}\label{eq_14}
\left\{\begin{array}{rl}
\partial_{tt} e_j^k-c ^2(\partial_{xx} e_j^k+\partial_{yy} e_j^k)-\lambda e_j^k(x, y, t-\tau )&=0, \ (x, y, t)\in \Omega_j\times (0, T),  \\ 
  e_j^k(x, y, t)&=0, \  (x, y, t)\in \Omega_j\times (-\tau, 0), \\ 
\partial_t   e_j^k(x, y, t)&=0, \  (x, y, t)\in \Omega_j\times (-\tau, 0), \\ 
\partial_x e_j^k(x_{j-1}, y, t)&=\partial_x e_{j-1}^k(x_{j-1}, y, t), \  t\in(0, T), \\
e_j^k(x_{j},y, t)&=\kappa^{k-1}_j(y,t) \  t\in(0, T),\\
e_j^k(x,0,t)&=e_j^k(x,\pi,t)=0, \ t \in (0,T).
\end{array}\right.
\end{equation}
\\
The update conditions along the interfaces are:\\
\begin{equation*}
\kappa_i^{k}(y,t)=\theta e_{i}^{k}(x_i,y,t) +(1-\theta) \kappa_i^{k-1}( y, t);\  1 \leq i < m \\
\end{equation*}
\begin{equation}\label{eq_15}
\kappa_j^{k}(y,t)=\theta e_{j+1}^{k}(x_j,y,t) +(1-\theta) \kappa_j^{k-1}(y, t);\ m\leq j \leq N-1.\\
\end{equation}
Applying a Fourier sine series in the $y$-direction now decouples our two-dimensional problems into a set of independent one-dimensional problems for each Fourier mode. Consequently, the solution is expressed as a Fourier sine series,
\[
e^k_i(x, y, t) = \sum_{n=1}^{\infty} E^k_i(x, n, t) \sin(ny),
\]
with the coefficients defined as,
\[
E^k_i(x, n, t) = \frac{2}{\pi} \int_0^{\pi} e^k_i(x, \eta, t) \sin(n \eta)\, d\eta.
\]
Thus, the problem reduces to
\begin{equation}\label{2d_to_1d}
\frac{\partial^2 E^k_i}{\partial t^2}(x, n, t) - c^2 \frac{\partial^2 E^k_i}{\partial x^2}(x, n, t) + c^2 n^2 E^k_i(x, n, t)-\lambda E_i^k(x,n, t-\tau ) = 0, 
\end{equation}
along with the appropriate boundary conditions. To prove the convergence result we first require the Lemma from \cite{mandal2026convergence} stated below.

\begin{lemma}\label{newsi}
For $\gamma, \beta, \lambda >0$ we have,
\begin{equation*}
  \begin{aligned}
\mathcal{L}^{-1}\left\{ e^{-\gamma \sqrt{s^2+\beta^2-\lambda e^{-\tau s}}}\right\} &= \left[ \delta(t - \gamma) - \frac{\beta \gamma J_1(\beta \sqrt{t^2 - \gamma^2})}{\sqrt{t^2 - \gamma^2}} \right] H(t - \gamma) \\
&+ \frac{\gamma \lambda}{2} J_0(\beta \sqrt{(t - \tau)^2 - \gamma^2}) H(t - \tau - \gamma) \\
&+ \frac{\gamma \lambda^2}{8} \left[ \frac{t - 2\tau - \gamma}{\beta(t - 2\tau + \gamma)} J_1(\beta \sqrt{(t - 2\tau)^2 - \gamma^2})\right. \\
&+ \left.\frac{\gamma}{2} J_0(\beta \sqrt{(t - 2\tau)^2 - \gamma^2}) \right] H(t - 2\tau - \gamma) + \dots\\
&=: \Psi(\gamma,\beta,\lambda,\tau,t)
\end{aligned}  
\end{equation*}
where $J_\nu(z)$ is the Bessel functions of the first kind.
\end{lemma}

\begin{theorem}[DNWR Convergence in 2D]
Let the relaxation parameter \( \theta = 1/2 \). For a fixed time window length \( T \in (0,\infty)\), the DNWR algorithm (\ref{eq_12})-(\ref{eq_15}) achieves convergence in at most \( k + 1 \) iterations, provided
$$
T/k <  d_{\min}/c,
$$
where \( c \) denotes the wave propagation speed.
\end{theorem}\label{thm2}

\begin{proof}
Applying the Laplace transform in time of equation (\ref{2d_to_1d}) results in,
\[
c^2 \frac{d^2 \hat{E}^k_i}{dx^2} =(s^2 + c^2 n^2-\lambda e^{-\tau s}) \, \hat{E}^k_i(x, s).
\]
For each Fourier mode 
$n$, we proceed as in the one-dimensional case treated in Theorem \ref{thm1}, where the interface functions satisfy recurrence relations of the form:

\begin{equation*}
\hat{h}^k_i(s) = \sum_j B^{(k)}_{ij}\left( \sqrt{s^2-\lambda e^{-\tau s}} \right) \ \, \hat{h}^0_j(s).    
\end{equation*}
When extending to two dimensions, the update conditions are reformulated separately for each Fourier mode 
\( n = 1,2,\ldots \), as follows:
\begin{equation}\label{updatestep2d}
    \hat \kappa^k_i(n, s) = \sum_j B^{(k)}_{ij}\left( \sqrt{s^2 + c^2 n^2-\lambda e^{-\tau s}} \right) \, \hat\kappa^0_j(n, s).
\end{equation}
Once again, we represent the interface values in the Laplace space as linear combinations of the initial data, weighted by coefficients
$B_{ij}^k\!\left( \sqrt{s^2 + c^2 n^2 - \lambda e^{-\tau s}} \right).$
In the one-dimensional case, the coefficients
$B^{(k)}_{ij}\!\left( \sqrt{s^2 - \lambda e^{-\tau s}} \right)$
are linear combinations of exponential functions of the form $e^{-\alpha \zeta}$, where
$\zeta = \frac{\sqrt{s^2 - \lambda e^{-\tau s}}}{c}.$
By applying the Efros theorem \eqref{efros} together with an exponential series expansion, we observe that the presence of the $\lambda$-term introduces an additional delay effect. Consequently, we retain only the leading term in the expansion and neglect higher-order contributions involving $\lambda$.
With this modification, the coefficients
$B^{(k)}_{ij}\!\left( \sqrt{s^2 + c^2 n^2 - \lambda e^{-\tau s}} \right)$
can be written as sums of exponential terms of the form
$e^{-\gamma \sqrt{s^2 + c^2 n^2 - \lambda e^{-\tau s}}},$
where $\gamma \ge \frac{k d_l}{c}$.

Using the definition of $\Psi(\gamma,\beta,\lambda,\tau,t)$ given in Lemma \ref{newsi}, we compute the inverse Laplace transform of \eqref{updatestep2d} and obtain
\begin{equation*}
\kappa_i^{k}(n,t)
=
\sum_{j} \sum_{l}
\Psi(\rho_{i,l,j,k}, cn, \lambda, \tau, t)
*
\kappa_j^{0}(n,t),
\end{equation*}
with $\rho_{i,l,j,k} \ge k d_{\min}/c$.
Therefore, for $t < k d_{\min}/c$, we have
$\kappa_i^{k}(n,t) = 0$
for every mode $n$.

Hence, whenever $t < k d_{\min}/c$, one additional iteration eliminates all interface discrepancies, and the DNWR algorithm recovers the exact solution throughout the entire spatial domain.
\end{proof}

 \section{Numerical Illustration}\label{sec_5}
 For validating our theoretical findings we implement DNWR algorithm using the Leapfrog scheme in both 1D and 2D spatial domain. The time discretization is chosen such that the delay parameter $\tau$ is expressed as an integer multiple of the time step $\Delta t$, i.e., $\tau = l \times\Delta t$. Under this assumption, the discrete scheme can be written as
\[
u_i^{n+1}
=
2u_i^n - u_i^{n-1}
+ \alpha^2 \left( u_{i+1}^n - 2u_i^n + u_{i-1}^n \right)
+ \Delta t^2 \lambda\, u_i^{n-l},
\]
where,
$\alpha = \frac{c\,\Delta t}{\Delta x}$
satisfies the CFL condition $\alpha \leq 1$. For all numerical experiments, the parameters are set to $c=1, \tau = 3$ and $\lambda=1.6$ and the initial interface guesses are taken as $h_i^0(t)=t^2, t\in(0,T]$, unless otherwise specified.

 \subsection{DNWR in 1D Mutisubdomain}
 We use discritization in space with step size $\Delta x=0.025$ and  in time with time step $\Delta t=0.025$. For numerical experiments we have divided spatial domain $\Omega=(0,6)$ into five subdomains where in the first case subdomains are taken as $\Omega_1=(0,1),\Omega_2=(1,2.5),\Omega_3=(2.5,3),\Omega_4=(3,5)$ and $\Omega_5=(5,6)$ with $h_{min}=0.5$; in the second case, the domain is partitioned into $\Omega_1=(0,1),\Omega_2=(1,2.5),\Omega_3=(2.5,3.5),\Omega_4=(3.5,5)$ and $\Omega_5=(5,6)$ with $h_{min}=1$. Error plots are included in Fig. \ref{diftime_wave_aar1}, \ref{diftime_wave_arr2} and \ref{diftime_wave_arr3} for arrangements 1, 2 and 3, respectively. Specifically, for a final time of $T=3$, wave speed $c=1$, and minimum subdomain size $d_{min}=0.5$, we calculate the ratio $cT/d_{min} = 6$. This yields $k=6$, implying that the convergence is achieved in at most $k+1 = 7$ iterations. \\
 We also run an experiment with different transmission conditions in different interfaces for arrangement~3 with $h_1^0=\tan(t)$, $h_2^0=\sin(t)$, $h_3^0=t^2$ and $h_4^0=t^3$ respectively, see Fig.\ref{diftime_wave_dif_hi}. We plot error curves in Fig. \ref{diffno_wave} for different numbers of subdomains (3, 5, 7 and 9) for time $T=12$ with $d_{min}=1$ in all cases and observe that maximum number of iteration required is at most $k+1$, consistent with the result established in Theorem \ref{thm1}.\\ 

 \begin{remark}
     Numerical experiments for arrangements 1 and 2 ensure that the convergence results are exactly the same across arrangements, although none of the arrangement is completely parallel. Arrangement~1 is sequential, while Arrangements 2 and 3 are partially parallel. For parallel implementation of Arrangement 3 for the classical heat equation, see \cite{Ong2016PipelineIO}.
 \end{remark}

\begin{figure}[!h]
    \centering
    \includegraphics[width=0.462\linewidth]{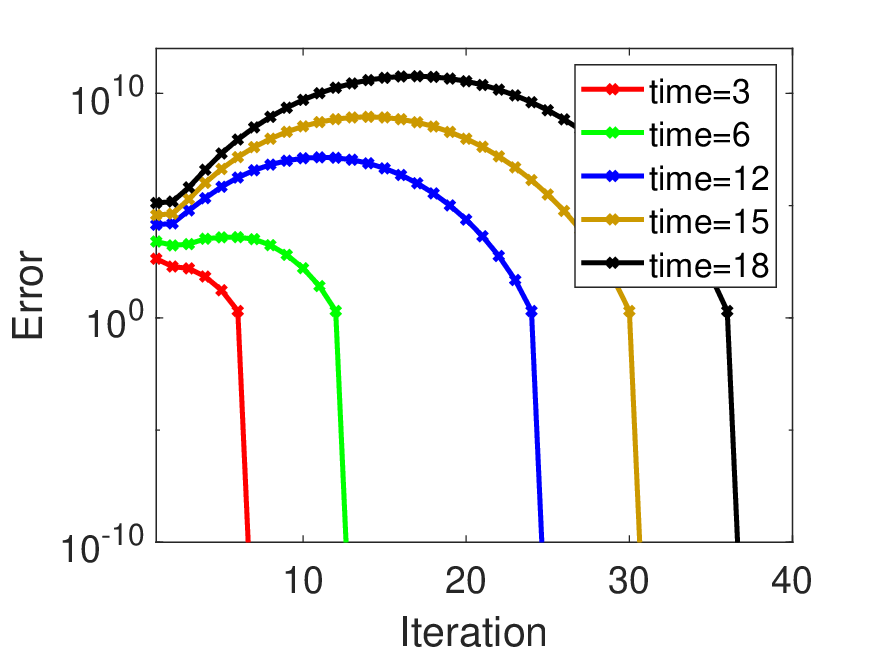}
    \includegraphics[width=0.462\linewidth]{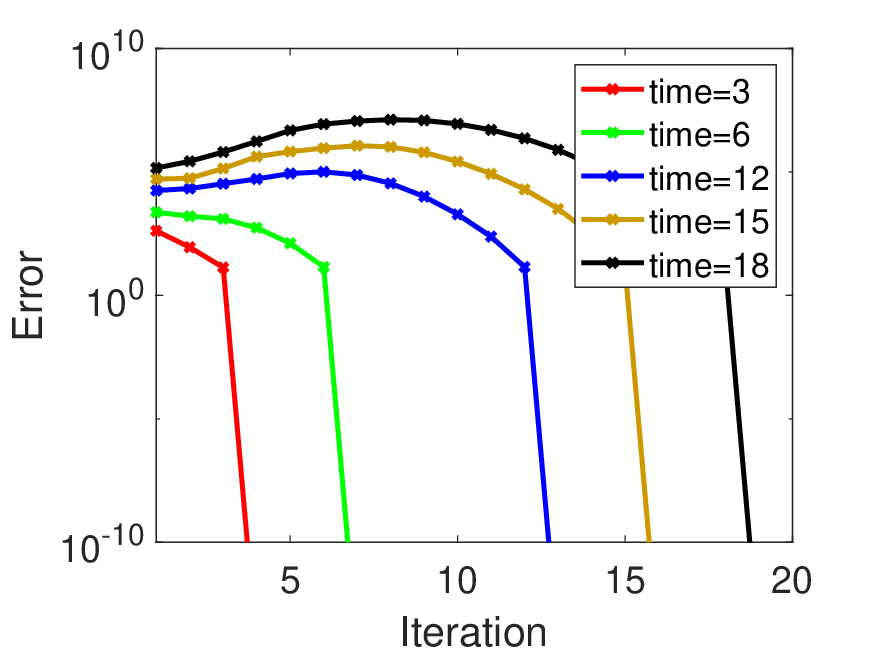}
    \caption{\textbf{Arrangement 1}: Convergence of DNWR methods for 5 subdomains. Left: Min subdomain length is $0.5$, Right: Min subdomain length is $1$.}
    \label{diftime_wave_aar1}
\end{figure}
\begin{figure}[!h]
    \centering
    \includegraphics[width=0.462\linewidth]{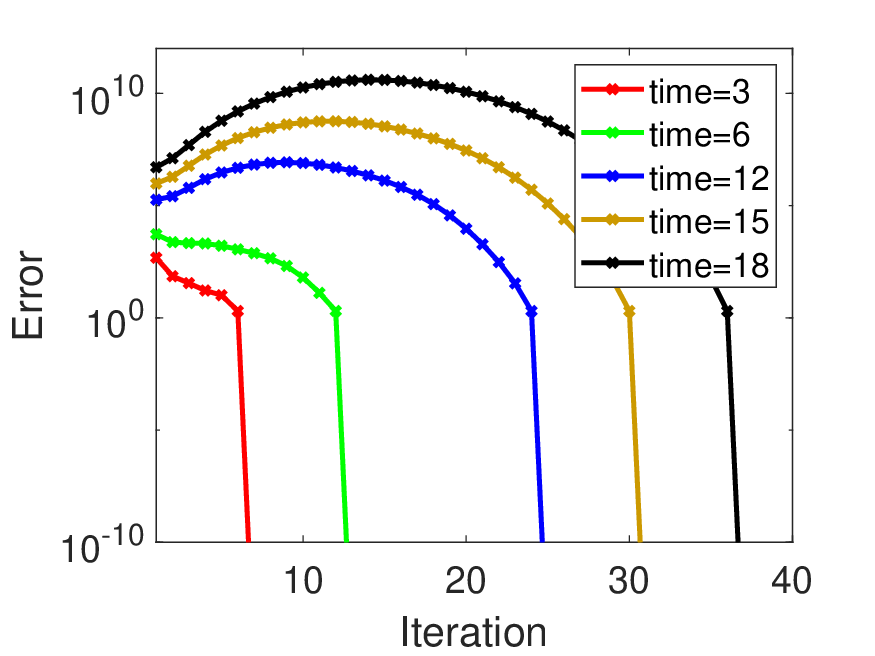}
    \includegraphics[width=0.462\linewidth]{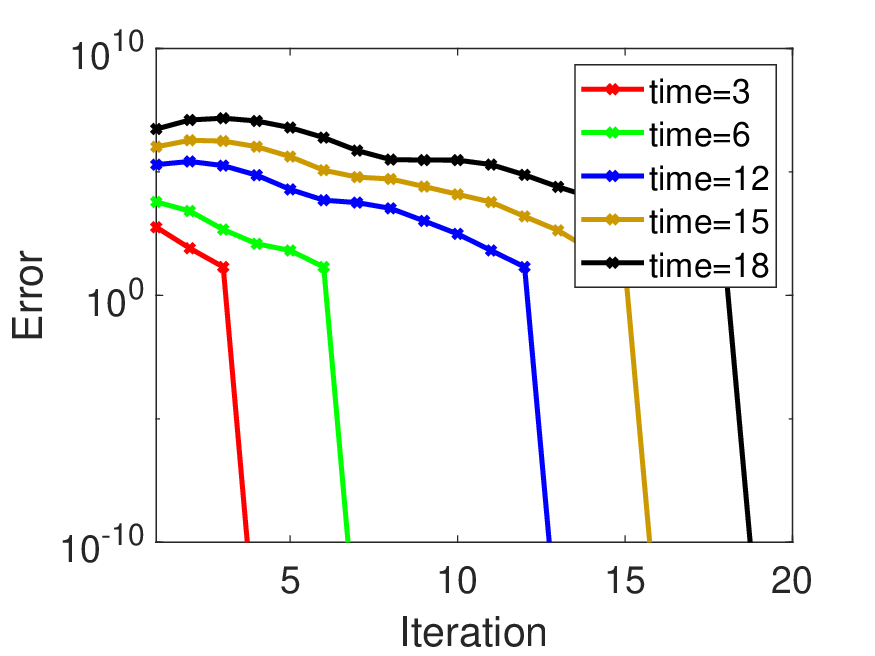}
    \caption{\textbf{Arrangement 2}: Convergence of DNWR methods for 5 subdomains. Left: Min subdomain length is $0.5$, Right: Min subdomain length is $1$.}
    \label{diftime_wave_arr2}
\end{figure}

 \begin{figure}[!h]
    \centering
    \includegraphics[width=0.462\linewidth]{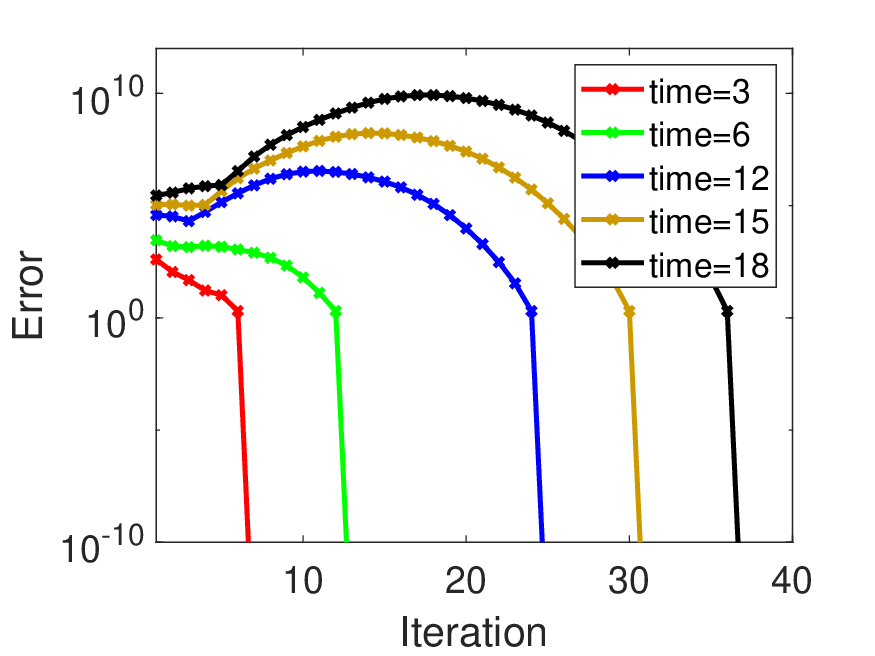}
    \includegraphics[width=0.462\linewidth]{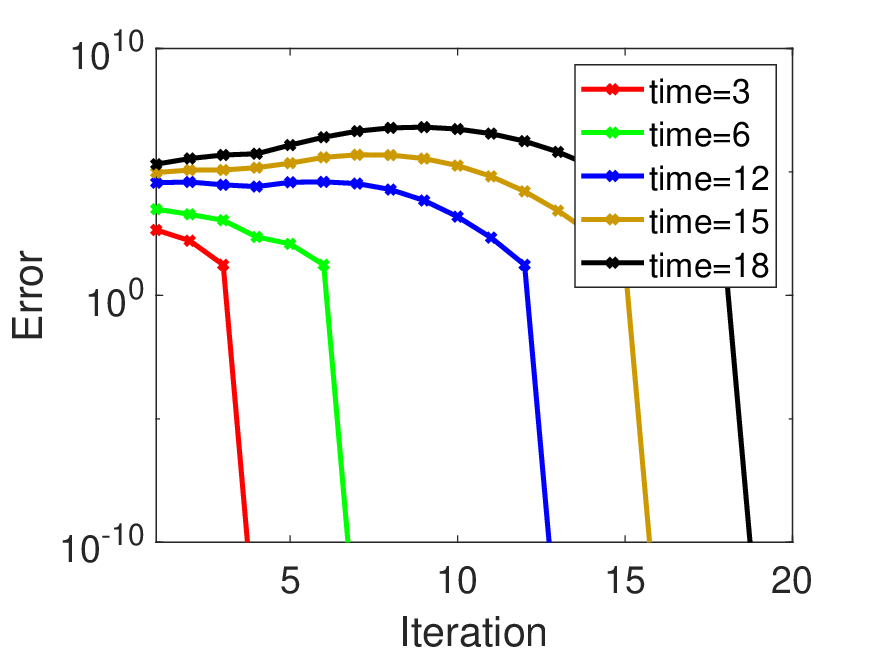}
    \caption{ \textbf{Arrangement 3}: Convergence of DNWR methods for 5 subdomains. Left: Min subdomain length is $0.5$, Right: Min subdomain length is $1$.}
    \label{diftime_wave_arr3}
\end{figure}
 \begin{figure}[!h]
    \centering
    \includegraphics[width=0.462\linewidth]{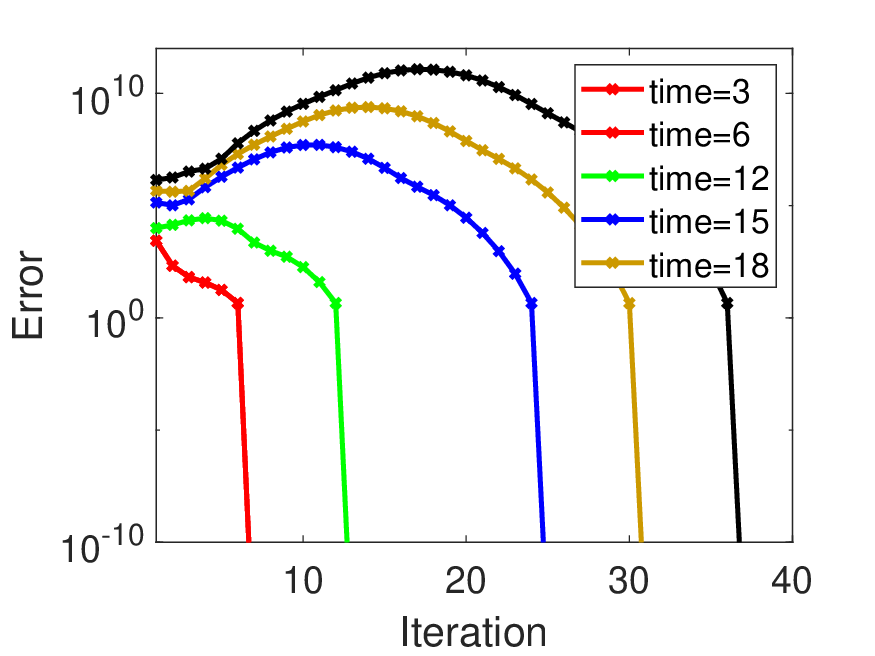}
    \includegraphics[width=0.462\linewidth]{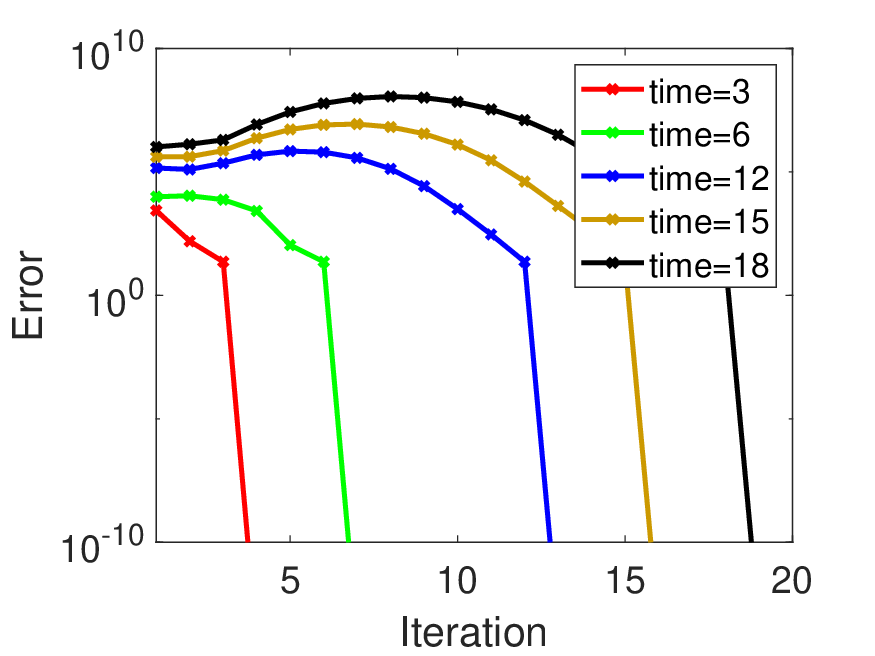}
    \caption{ \textbf{Arrangement 3}: Convergence of DNWR methods
    with different interface conditions. Left: Min subdomain length is $0.5$, Right: Min subdomain length is $1$.}
    \label{diftime_wave_dif_hi}
\end{figure} 
\begin{remark}
    We obtain the condition $T/k < d_{\min}/c$ for both one- and two-dimensional problems in a multi-subdomain setting. Notably, this condition is independent of the delay term $\tau$. Therefore, the magnitude of the delay does not influence the convergence behavior (see Fig.~\ref{difdelay_wave_arr2}). 
\end{remark}

\begin{figure}[!h]
    \includegraphics[width=0.462\linewidth]{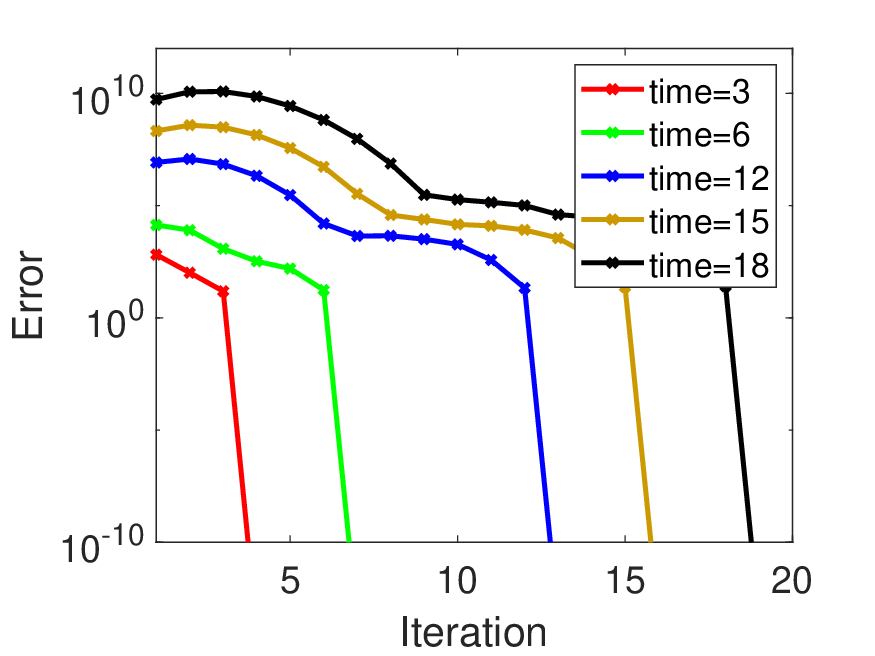}
    \includegraphics[width=0.462\linewidth]{dnwrarr2_5subhmin_one.eps}
    \caption{\textbf{Arrangement 3}: Convergence of DNWR methods for 5 subdomains with $d_{min}=1$. Left:  $\tau$ is $0.3$; Right: $\tau$ is $3$.}
    \label{difdelay_wave_arr2}
\end{figure}

\begin{figure}
    \centering
    \includegraphics[width=0.462 \linewidth]{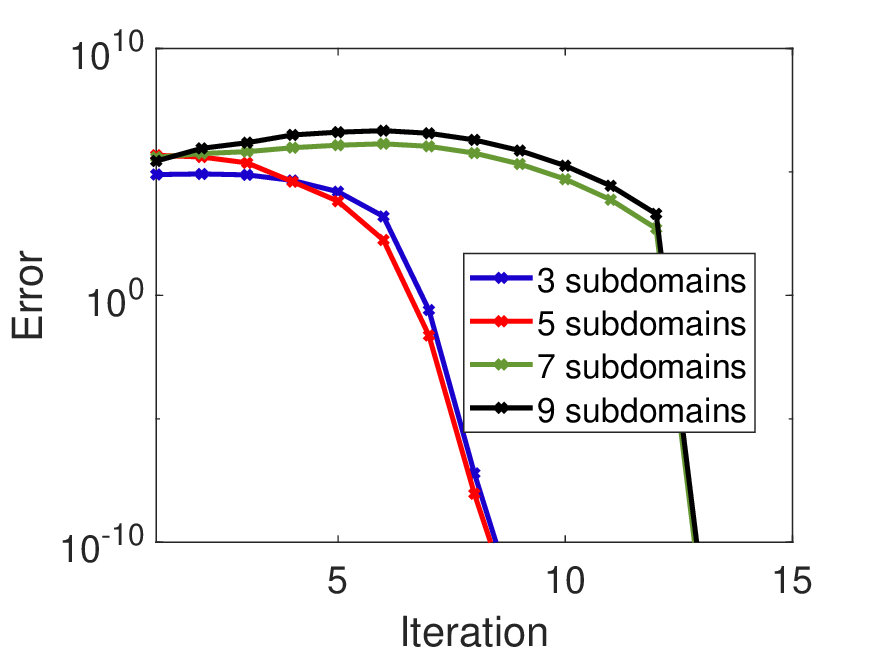}
    \includegraphics[width=0.462 \linewidth]{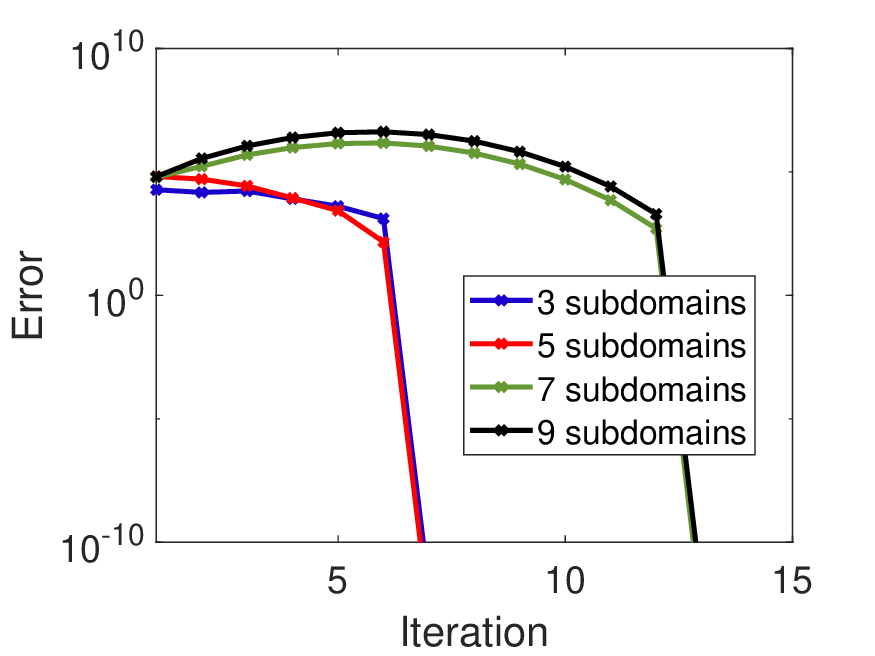}
    \caption{DNWR for multiple subdomains for Arrangement 3 $d_{min}=1$ and $T=12.$. Left: $\tau=1$, Right: $\tau=3$.}
     \label{diffno_wave}
    \end{figure}

\subsection{DNWR in 2D Multisubdomain}
For the implementation of DNWR in 2D, the mesh size are taken as $\Delta x =0.1=\Delta y$ and the time step is $\Delta t=0.025.$ The spatial domain $\Omega=(0,6)\times(0,6)$ is divided into three subdomains.\\
\textbf{Case I -} We considered $\Omega_1=(0,3)\times(0,6)$, $\Omega_2=(3,5.5)\times(0,6)$ and $\Omega_3 =(5.5,6)\times (0,6)$ with minimum subdomain length $d_{min}=0.5$.\\
\textbf{Case II -} We considered $\Omega_1=(0,3.5)\times (0,6),$ $\Omega_2=(3.5,5)\times (0,6)$ and $\Omega_3=(5,6)\times (0,6)$ with $d_{min}=1.$ (See Fig. \ref{diftime_wave2D_dif_hi}.)

\begin{figure}[!h]
    \centering
    \includegraphics[width=0.462\linewidth]{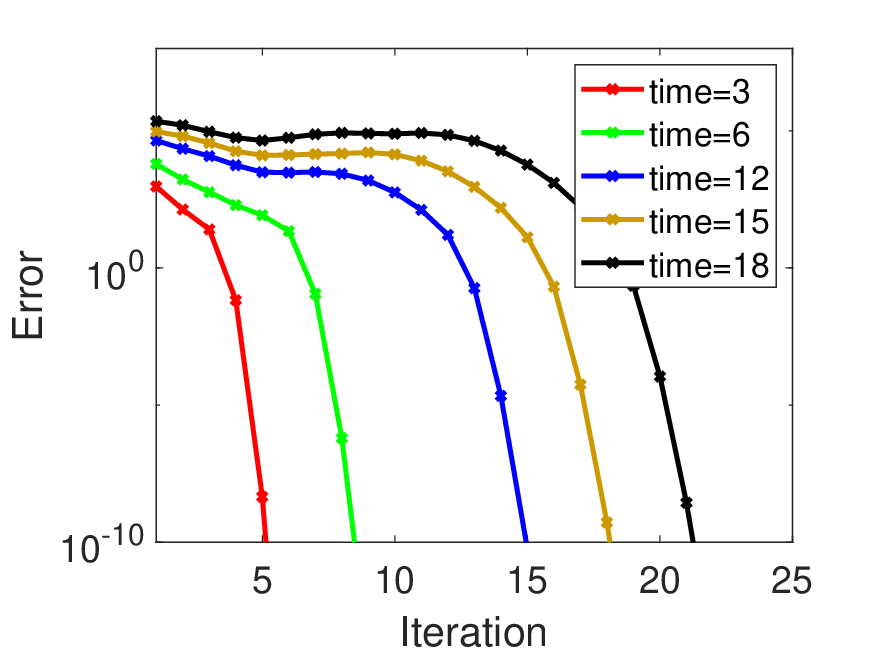}
     \includegraphics[width=0.462\linewidth]{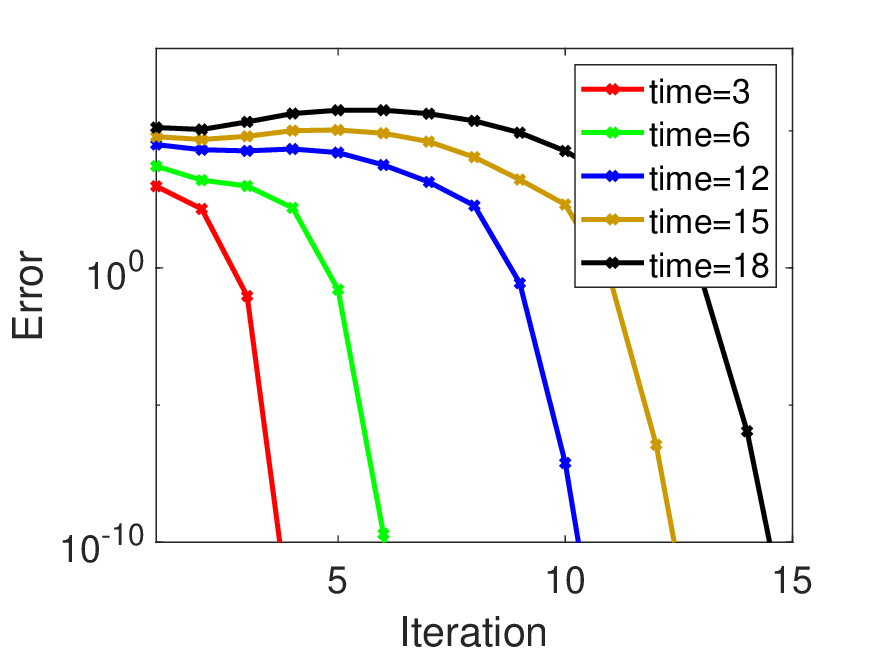}
    \caption{ \textbf{Arrangement 3}: Convergence of DNWR methods
    with different interface conditions in 2D. Left: Min subdomain length is $0.5$, Right: Min subdomain length is $1$.}
       \label{diftime_wave2D_dif_hi}
\end{figure}
\subsection{Comparative Study}
We evaluated the performance of DNWR and NNWR against both classical and optimized SWR methods using a three-subdomain configuration. The experimental results indicate that both DNWR and NNWR achieve faster convergence than classical SWR in both 1D and 2D scenarios; moreover, classical SWR converges only with an overlap. In the 1D case, the convergence rate of the optimized SWR is comparable to that of DNWR and NNWR. However, in 2D simulations, both DNWR and NNWR significantly outperform the optimized SWR approach. 

For these experiments, the overlap is set to $4 \times \Delta x$ and the numerical error curves are shown in Fig.~\ref{compare_wave}. For optimised SWR we are using absorbing boundary conditions \cite{OSWR1Dwave} as the transmission condition. The parameters considered are: a delay value of $\tau = 3$, a time window length of $T = 3$, and a minimum subdomain width of $d_{\min} = 1$. The discretization parameters are configured as follows:
\begin{itemize}
    \item \textbf{In 1D:} $\Delta t = 0.1$ and $\Delta x = 0.1$. The spatial domain $\Omega = (0,6)$ is partitioned into three subdomains: $\Omega_1 = (0,1)$, $\Omega_2 = (1,3)$, and $\Omega_3 = (3,6)$.
    \item \textbf{In 2D:} $\Delta t = 0.07$ and $\Delta x = \Delta y = 0.1$. The spatial domain $\Omega = (0,6) \times (0,6)$ is divided into $\Omega_1 = (0,1) \times (0,6)$, $\Omega_2 = (1,3) \times (0,6)$, and $\Omega_3 = (3,6) \times (0,6)$.
\end{itemize}
Table \ref{tab:comparison} summarises the convergence properties of different WR methods for the wave equation (even with delay).

\begin{table}[h]
\centering
\caption{Comparison of WR methods}
\label{tab:comparison}
\begin{tabular}{lcc}
\toprule
Method & 2 subdomains & Many subdomains \\
\midrule
Classical SWR (overlap $\delta$) & $k > \dfrac{cT}{\delta}$ & $k > \dfrac{cT}{\delta_{\min}}$ \\
Optimized SWR (exact BC) & Finite iterations & Finite iterations \\
DNWR ($\theta=1/2$) & $k > \dfrac{cT}{2\min\{a,b\}}$ & $k > \dfrac{cT}{d_{\min}}$ \\
NNWR ($\theta=1/4$) & $k > \dfrac{cT}{4\min\{a,b\}}$ & $k > \dfrac{cT}{2d_{\min}}$ \\
\bottomrule
\end{tabular}
\end{table}
In practice, $\delta$ is usually chosen as a small fraction of the subdomain size to minimize computational overhead. The table shows that DNWR and NNWR converge in fewer iterations.

\begin{figure}[!h]
    \centering
    \subfloat{\includegraphics[width=0.462\linewidth]{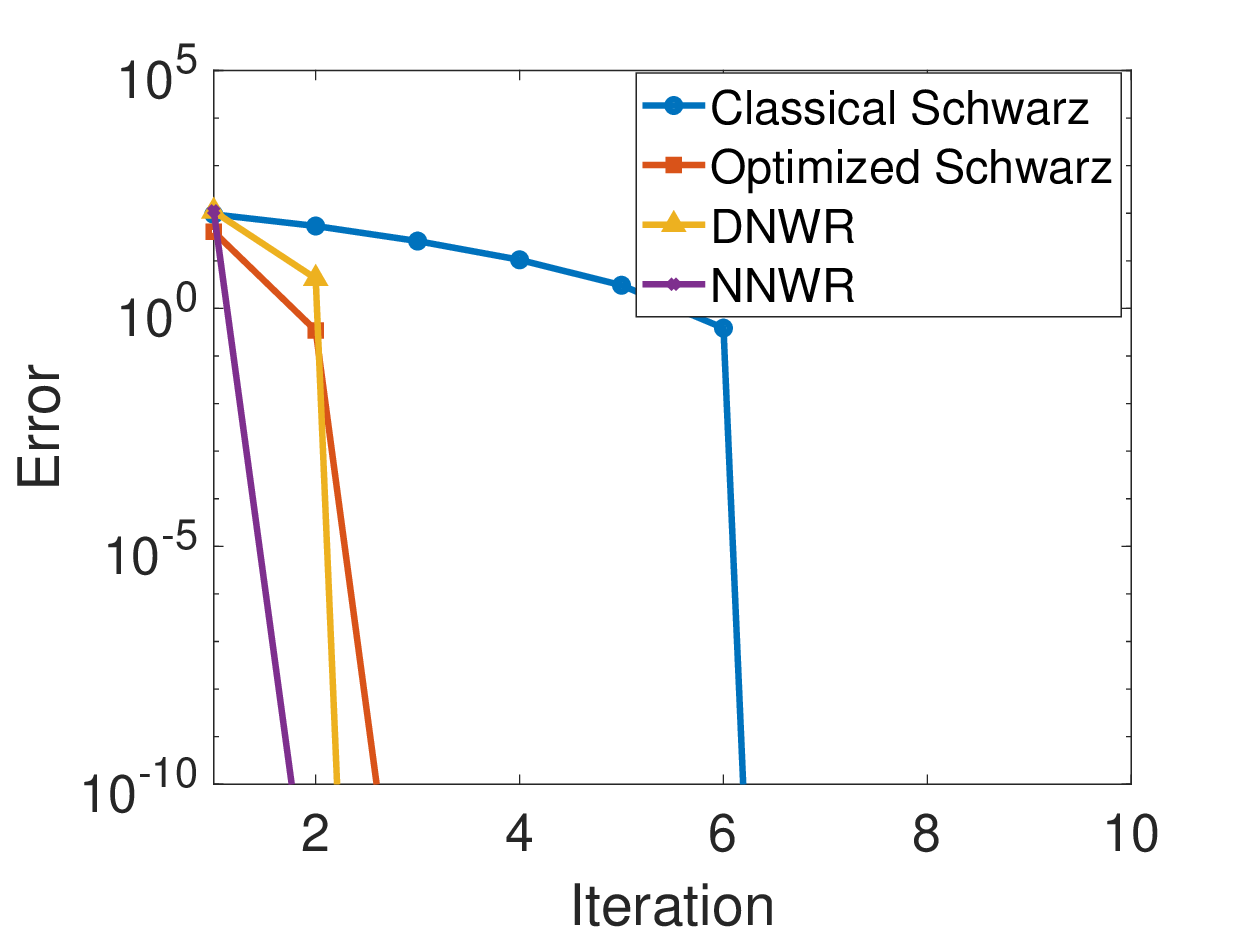}}
    \qquad
    \subfloat{\includegraphics[width=0.462\linewidth]{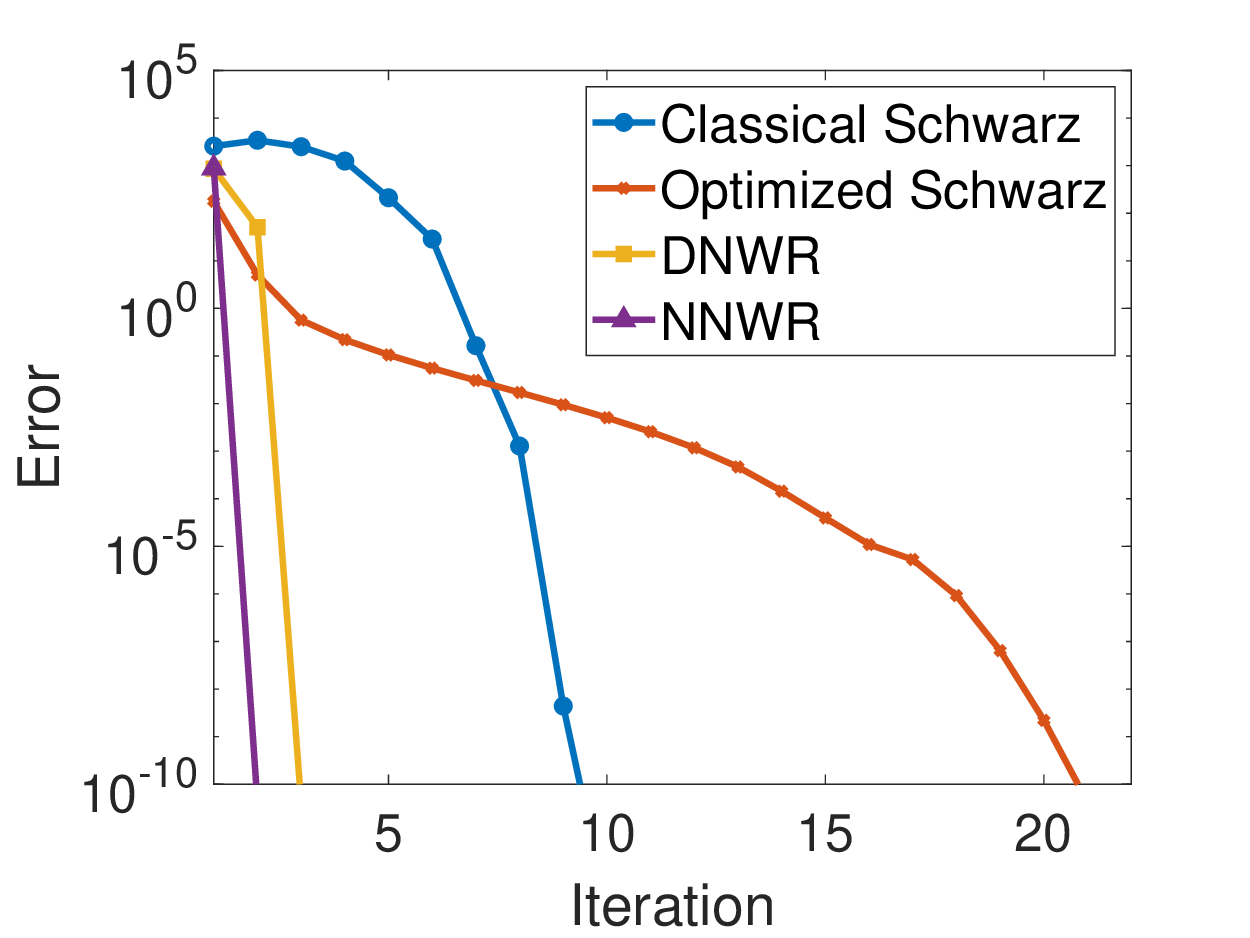}}
    \caption{ Comparison of DNWR and NNWR with Classical SWR and Optimised SWR, Left: 1D 3 subdomain; Right: 2D 3 subdomain.}
       \label{compare_wave}
\end{figure}
\section{Conclusion}\label{sec_6}
We adapted the DNWR algorithm for a multi-domain configuration to solve time-delayed hyperbolic PDEs. Through Laplace-transform analysis in both 1D and 2D spatial domains, we proved that the DNWR algorithm converges in finitely many steps for multi-subdomain hyperbolic PDEs with time delays. For the optimal parameter $\theta = 1/2$, we showed the maximum number of iterations required to reach convergence is:
\begin{equation*}
k_{\text{conv}} = \left\lceil \frac{cT}{d_{\min}} \right\rceil+1.
\end{equation*}
This theoretical convergence is reinforced by numerical experiments under various structural configurations, which consistently attest to the algorithm's robustness. The results indicate that DNWR converges in fewer iterations than the classical SWR method. In comparison with DNWR, the NNWR method demonstrates improved convergence for larger time intervals. However, this advantage comes at the cost of increased computational effort, as NNWR requires roughly twice the work due to the additional Dirichlet and correction steps.
\section*{Acknowledgments}
The authors would like to thank IIT Bhubaneswar for the research facility.
\section*{Contribution}
All authors contributed equally.
\section*{Conflicts of Interest}
The authors declare no conflicts of interest.
\section*{Data availability}
All the data that were produced or generated during the course of the research have been included in the manuscript.
\bibliography{sn-bibliography}

\clearpage

\end{document}